\documentclass[a4paper, 12pt]{article}
\usepackage{mathrsfs}
\usepackage{amsmath}
\usepackage{amsfonts}
\usepackage[T1]{fontenc}
\usepackage[latin9]{inputenc}
\usepackage{amsthm}
\usepackage{graphicx}
\usepackage{amsmath}

\usepackage{amsthm}
\usepackage{array}
\usepackage{cases}
\makeatletter
\def\th@plain{%
  \upshape 
}
\makeatother

\makeatletter
\renewenvironment{proof}[1][\proofname]{\par
  \pushQED{\qed}%
  \normalfont \topsep6\p@\@plus6\p@\relax
  \trivlist
  \item[\hskip\labelsep
        \bfseries
    #1\@addpunct{.}]\ignorespaces
}{%
  \popQED\endtrivlist\@endpefalse
}
\makeatother

\newtheorem{theorem}{Theorem}[section]

\newtheorem{conjecture}[theorem]{Conjecture}
\numberwithin{equation}{section}

\newtheorem{thm}{Theorem}[section]

\newtheorem{lem}[thm]{Lemma}

\newtheorem{conj}[thm]{Conjecture}

\newtheorem{clm}{Claim}

\numberwithin{equation}{section}

\usepackage[varg]{txfonts}
\usepackage{graphicx, epsfig, subfigure}
\usepackage{enumerate} 
\usepackage[square, numbers, sort&compress]{natbib} 
\ifx\pdfoutput\undefined
 \usepackage[dvipdfm,%
  pdfstartview=FitH, 
  bookmarks=true,%
  bookmarksnumbered=true, 
  bookmarksopen=true, 
  plainpages=false,%
  pdfpagelabels,%
  colorlinks=true, 
  linkcolor=blue, 
  citecolor=blue,%
  urlcolor=black,
  pdfborder=001]{hyperref}
  \AtBeginDvi{}  
\else
 \usepackage[pdftex,%
  pdfstartview=FitH, 
  bookmarks=true,%
  bookmarksnumbered=true, 
  bookmarksopen=true, 
  plainpages=false,%
  pdfpagelabels,%
  colorlinks=true, 
  linkcolor=blue, 
  citecolor=blue,%
  urlcolor=black,
  pdfborder=001]{hyperref}
\fi

\numberwithin{equation}{section}

\newcommand{\gx}{G^{\times}}

\begin{document}

\title{\LARGE A structure of 1-planar graph and its applications to coloring problems
\thanks{Supported by the Fundamental Research Funds for the Central Universities (No.\,JB170706), the Natural Science Basic Research Plan in Shaanxi Province of China (No.\,2017JM1010), and
the National Natural Science Foundation of China (Nos.\,11871055, 11301410).}}
\author{Xin Zhang$^1$\thanks{Corresponding author. Email: xzhang@xidian.edu.cn},\,\, Bei Niu$^1$,\,\, Jiguo Yu$^2$\\
{\small 1. School of Mathematics and Statistics, Xidian University, Xi'an, Shaanxi, 710071, China}\\
{\small 2. School of Information Science and Engineering, Qufu Normal Univeristy, Rizhao, Shandong, 276826, China}}
\date{}

\maketitle

\begin{abstract}\baselineskip 0.60cm
A graph is 1-planar if it can be drawn on a plane so that each edge is crossed by at most one other edge. In this paper, we first give a useful structural theorem for 1-planar graphs, and then apply it to the list edge and list total coloring, the $(p,1)$-total labelling, and the equitable edge coloring of 1-planar graphs. More precisely, we verify the well-known List Edge Coloring Conjecture and List Total Coloring Conjecture for 1-planar graph with maximum degree at least 18, prove that the $(p,1)$-total labelling number of every 1-planar graph $G$ is at most $\Delta(G)+2p-2$ provided that $\Delta(G)\geq 8p+2$ and $p\geq 2$, and show that every 1-planar graph has an equitable edge coloring with $k$ colors for any integer $k\geq 18$. These three results respectively generalize the main theorems of three different previously published papers.

\vspace{3mm}\noindent \emph{Keywords:  1-planar graph; list edge coloring; list total coloring; $(p,1)$-total labelling; equitable edge coloring}.
\end{abstract}

\baselineskip 0.60cm

\section{Introduction }

Throughout the paper, all graphs are finite, simple and undirected. By $V(G),E(G),\delta(G)$ and $\Delta(G)$, we denote the set of vertices, the set of edges, the minimum degree and the maximum degree of a graph $G$. If $G$ is a plane graph, then $F(G)$ denotes the set of faces of $G$. A $k$-, $k^+$- and $k^-$\emph{-vertex} (resp.\,\emph{face}) is a vertex (resp.\,face) of degree $k$, at least $k$ and at most $k$, respectively.  For undefined concepts we refer the reader to
\cite{Bondy}.

A \emph{proper edge (resp.\,total) $k$-coloring} of $G$ is a function $\varphi$ from $E(G)$ (resp.\,$V(G)\cup E(G)$) to $\{1,2,\ldots,k\}$ so that $\varphi(x)\neq \varphi(y)$ if $x$ and $y$ are two adjacent edges (resp.\,adjacent/incident elements) in $G$. The minimum $k$ such that $G$ has a proper edge (resp.\,total) $k$-coloring is the \emph{edge (resp.\,total) chromatic number} of $G$, denoted by $\chi'(G)$ (resp.\,$\chi''(G)$).

An \emph{edge assignment} $L$ for the graph $G$ is a function so that for any edge $e\in E(G)$, $L(e)$ is a list of possible colors that can be used on $e$.
If $G$ has a proper edge coloring $\varphi$ such that $\varphi(e)\in L(e)$ for each edge $e$ of $G$, then we say that $G$ is \emph{edge-$L$-colorable} and $\varphi$ is an \emph{edge-$L$-coloring} of $G$. A graph $G$ is $edge$ $f$-$choosable$ if, whenever we give lists $L(e)$ of $f(e)$ colors (where $f$ is a function from $E(G)$ to $\mathbb{N}$) to each edge $e$ of $G$, $G$ is edge-$L$-colorable. If $G$ is edge $f$-choosable and $f(e)=k$ for each edge $e\in E(G)$, then $G$ is  \emph{edge $k$-choosable}. The minimum $k$ such that $G$ is edge $k$-choosable is the \emph{list edge chromatic number} or \emph{edge choosability} of $G$, denoted by $\chi'_l(G)$. The \emph{list total chromatic number} or \emph{total choosability} of $G$, denoted by $\chi''_l(G)$, is defined similarly.

Concerning the edge choosability and the total choosability of graphs, there are two well-known conjectures.

\begin{conjecture}[List Edge Coloring Conjecture]
$\chi'_{l}(G)=\chi'(G)$ for any graph $G$.
\end{conjecture}

\begin{conjecture}[List Total Coloring Conjecture]
$\chi''_{l}(G)=\chi''(G)$ for any graph $G$.
\end{conjecture}

The List Edge Coloring Conjecture (LECC) was independently posed by Vizing, and by Gupta, and by Albertson and Collins, and by Bollob\'as
and Harris (see \cite{Toft} for the history of this problem). The List Total Coloring Conjecture (LTCC) was posed by Borodin, Kostochka and Woodall \cite{Borodin.lcc}. Until now, the above two conjectures are still widely open, and particular research on some special but nontrivial classes of graphs is carried on. For example, Borodin, Kostochka and Woodall \cite{Borodin.lcc} proved in 1997 that LECC and LTCC hold for planar graphs with maximum degree at least 12. Although this is a result of two decades ago, the bound 12 for the maximum degree there is still the best known bound at this moment.


The aim of this paper is to study these conjectures for the family of 1-planar graphs. A graph is \emph{1-planar} if it can be drawn on a plane so that each edge is crossed by at most one other edge, and this drawing is a \emph{1-plane graph}. Usually, the number of crossings in a 1-plane graph is assumed to be as few as possible. The notion of 1-planarity was introduced by Ringel \cite{R65} while trying to simultaneously color the vertices and faces of a plane graph such that any pair of adjacent or incident elements receive different colors. Ringel \cite{R65} proved that every 1-planar graph is $7$-colorable, and this bound for the chromatic number was later improved to 6 (being sharp) by Borodin \cite{B84,B95}. Recently in 2017, Kobourov, Liotta and Montecchiani \cite{KLM17} reviewed the current literature covering
various research streams about 1-planarity, such as characterization and recognition, combinatorial
properties, and geometric representations.

For the edge and the total colorings of 1-planar graphs, Zhang and Wu \cite{ZW11} proved that the edge chromatic number of every 1-planar graph with maximum degree $\Delta\geq 10$ is equal to $\Delta$, and Zhang and Liu \cite{ZL12} conjectured that the bound for $\Delta$ can be lowered to 8, which is best possible. Zhang, Hou and Liu \cite{ZHL15} proved that the total chromatic number of every 1-planar graph with maximum degree $\Delta\geq 13$ is at most $\Delta+2$. In 2012, Zhang, Wu and Liu \cite{ZWL12} proved the following theorem, which confirms LECC and LTCC for 1-planar graphs with large maximum degree.

\begin{thm}\cite[Zhang, Wu and Liu]{ZWL12}\label{old-1}
  If $G$ is a 1-planar graph with maximum degree $\Delta\geq 21$, then $\chi'(G)=\chi'_l(G)=\Delta$ and $\chi''(G)=\chi''_l(G)=\Delta+1$.
\end{thm}

 A \emph{$(p,1)$-total $k$-labelling} of a graph $G$, introduced by Havet and Yu \cite{HY.2002,H.2003}, is a function $f$ from $V(G)\cup E(G)$ to the color set $\{0,1,\cdots,k\}$ such that $|f(u)-f(v)|\geq 1$ if $uv\in E(G)$, $|f(e_1)-f(e_2)|\geq 1$ if $e_1$ and $e_2$ are two adjacent edges in $G$, and $|f(u)-f(e)|\geq p$ if the vertex $u$ is incident to the edge $e$.
The minimum $k$ such that $G$ has a $(p,1)$-total $k$-labelling, denoted by $\lambda_p^T(G)$, is  the {\it $(p,1)$-total labelling number} of $G$. It is easy to see that $\lambda_1^T(G)=\chi''(G)-1$. Havet and Yu \cite{HY.2002,Havet} put forward the following conjecture.

\begin{conj}[$(p,1)$-Total Labelling Conjecture]
$\lambda^T_p(G)\leq \min\{\Delta(G)+2p-1,2\Delta(G)+p-1\}$.
\end{conj}

For $p=1$, the above conjecture is nothing but the well-known Total Coloring Conjecture, which states that $\chi''(G)\leq \Delta(G)+2$. Since $\Delta(G)+1$ is a natural lower bound for $\chi''(G)$, and the $(p,1)$-total labelling is a generalization of the total coloring, it is interesting to consider when we have $\lambda^T_p(G)\leq\Delta(G)+2p-2$. Concerning this problem, Bazzaro, Montassier and Raspaud \cite{Bazzaro} proved that if $G$ is a planar graph with $\Delta(G)\geq 8p+2$ and $p\geq 2$, then $\lambda^T_p(G)\leq\Delta(G)+2p-2$. The lower bound for the maximum degree in this result was recently improved to $4p+4$ by Sun and Wu \cite{SW17}. For 1-planar graphs, Zhang, Yu and Liu \cite{ZYL11} proved the following result.

\begin{thm}\cite[Zhang, Yu and Liu]{ZYL11}\label{old-2}
  If $G$ is a 1-planar graph with $\Delta(G)\geq 8p+4$ and $p\geq 2$, then $\lambda^T_p(G)\leq\Delta(G)+2p-2$.
\end{thm}

Let $\varphi$ be a function from $E(G)$ to $\{1,2,\ldots,k\}$.
For each vertex $v\in V(G)$, let $c_i(\varphi ,v)= |\{uv\in E(G)\,|\,\varphi (uv)=i\}|$.
An edge $k$-coloring $\varphi$ is $equitable$ if for each $v \in V(G)$, we have
$$
|c_i(\varphi, v)-c_j(\varphi, v)| \leq 1 \;\;\;\; (1 \leq i<j \leq k).%
$$
The \emph{equitable edge chromatic number} $\chi'_{=}(G)$ of a graph $G$ is the smallest number $k$ such that $G$ has an equitable edge $k$-coloring. However, the notion $\chi'_{=}(G)$ is somehow trivial since every graph has an equitable edge $1$-coloring. Therefore, we need another notion to characterize the equitability of an edge coloring .

The \emph{equitable edge chromatic threshold} $\chi'_{\equiv}(G)$ of $G$ is the smallest $k$ such that $G$ has an equitable edge $k'$-coloring for any $k'\geq k$. For example, the equitable edge chromatic threshold of any odd cycle is exactly 3.

From the above definitions, one can easily find that a proper edge coloring of $G$ is trivially equitable. Hence we immediately conclude that $\chi'_{\equiv}(G)\leq \chi'(G)$. However, $\chi'(G)$ may be a too large upper bound for $\chi'_{\equiv}(G)$. For example, Song, Wu and Liu \cite{Song} proved for series-parallel graphs $G$ that $\chi'_{\equiv}(G)=1$ if and only if $G$ is not a connected graph with the number of edges being odd in which each vertex has even degree. Hu et al.\,\cite{HWYZ17} proved that $\chi'_{\equiv}(G)\leq 12$  for any planar graph $G$. For 1-planar graphs,  Hu et al.\,\cite{HWYZ17} gave the following result.

\begin{thm}\cite[Hu et al.]{HWYZ17}\label{old-3}
  If $G$ is a 1-planar graph, then $\chi'_{\equiv}(G)\leq 21$.
\end{thm}

In this paper, we first present in Section 2 an useful structural theorem for 1-planar graphs, which can be used to consider not only the list edge and list total coloring problems, but also some other coloring problems such as the $(p,1)$-total labelling and the equitable edge coloring.
In Section 3, we prove that LECC and LTCC hold for 1-planar graphs with maximum degree at least 18, which improves Theorem \ref{old-1}.
In Section 4, we consider the $(p,1)$-total labeling of 1-planar graph $G$ by proving $\lambda^T_p(G)\leq\Delta(G)+2p-2$ if $\Delta(G)\geq 8p+2$ and $p\geq 2$. This improves Theorem \ref{old-2}. Actually, this result also generalizes the previously mentioned result of Bazzaro, Montassier and Raspaud on planar graphs to the same result on 1-planar graphs.
In Section 5, we improve the upper bound for the equitable edge chromatic threshold of 1-planar graphs in Theorem \ref{old-3} to 18.



\section{Structural Theorem}

The \emph{associated plane graph} $\gx$ of a 1-plane graph $G$ is the plane graph that is obtained from $G$ by turning
all crossings of $G$ into new vertices of degree four. These new vertices in $\gx$ are \emph{false} vertices, and the original vertices of $G$ are \emph{true} ones. A face in $\gx$ is \emph{false} if it is incident with at least one false vertex, and \emph{true} otherwise.

\begin{lem}\label{n.adj}\cite[Lemma 1]{ZW11}
If $G$ is a $1$-plane graph, then

(a) false vertices in $\gx$ are not adjacent;

(b) false 3-face in $\gx$ is not incident with 2-vertex;

(c) if a $3$-vertex $v$ is incident with two
$3$-faces and adjacent to two false vertices in $\gx$, then $v$ is incident with a $5^+$-face;

(d) there exists no edge $uv$ in $\gx$ such that $d_{\gx}(u)=3$, $v$ is a false vertex, and $uv$ is incident with two $3$-faces.
\end{lem}

A bipartite subgraph $F$ of $G$ is a \emph{$k$-alternator} of $G$ with partite sets $X, Y$ for some $2\leq k\leq \lfloor\frac{\Delta(G)}{2}\rfloor$ if $d_F(x)=d_G(x)\leq k$ for each $x\in X$, and $d_F(y)\geq d_G(y)+k-\Delta(G)$ for each $y\in Y$.

A bipartite subgraph $F$ of $G$ is a \emph{$k$-alternating subgraph} of $G$ with partite sets $X, Y$ for some $2\leq k\leq \lfloor\frac{\Delta(G)}{2}\rfloor$ if $d_F(x)=d_G(x)\leq k$ for each $x\in X$, and $d_F(y)\geq k$ for each $y\in Y$.

\begin{lem}\cite[Lemma 2.4]{WW08} (resp.\,\cite[Lemma 7]{HWYZ17})\label{master}
Let $2\leq k\leq \lfloor\frac{\Delta}{2}\rfloor$ be a fixed integer and let $G$ be a graph without $k$-alternator (resp.\,$k$-alternating subgraph).
Let $X_k=\{x\in V(G)~|~d_G(x)\leq k\}$ and $Y_k=\bigcup_{x\in X_k}N_G(x)$. If
$X_k\not=\emptyset$, then there exists a bipartite subgraph $M_k$ of $G$ with partite sets $X_k, Y_k$ such that $d_{M_k}(x)=1$ for each $x\in
X_k$ and $d_{M_k}(y)\leq k-1$ for each $y\in Y_k$.
\end{lem}

\noindent \textbf{Remark:} The second result (while $k$-alternating subgraph is forbidden in $G$) of the above lemma comes from the first three paragraphs of the proof of Lemma 7 in \cite{HWYZ17}. Although $k$ is assumed to be at most 5 in \cite{HWYZ17}, the upper bound for $k$ can actually be relaxed to $\lfloor\frac{\Delta}{2}\rfloor$ without changing any word in their proof.

Following Lemma \ref{master}, we call $y$ the $k$-$master$ of $x$ if $xy\in M_k$ and $x\in X_k$. By Lemma \ref{master}, we conclude that
\begin{align}\label{1}
each~d\texttt{-}vertex~(2\leq d\leq \big\lfloor\frac{\Delta}{2}\big\rfloor)~has~a~k\texttt{-}master~for~each~d\leq k\leq \big\lfloor\frac{\Delta}{2}\big\rfloor
\end{align}
and
\begin{align}\label{2}
each ~vertex ~of~G~may ~be ~a ~k\texttt{-}master ~(2\leq k\leq \big\lfloor\frac{\Delta}{2}\big\rfloor)~of ~at ~most ~k-1 ~vertices.
\end{align}

\begin{thm}\label{str-1-g3}
If $G$ is a 1-planar graph with minimum degree at least 2,
then $G$ contains

(a) an edge $xy$ with $d_G(x)\leq 5$ and $d_G(x)+d_G(y)\leq 19$, or

(b) an edge $xy$ with $d_G(x),d_G(y)\geq 6$ and $d_G(x)+d_G(y)\leq 16$, or

(c) a $k$-alternator (resp.\,$k$-alternating subgraph) for some $k\in \{2,3,4,5\}$.
\end{thm}

\begin{proof}
Suppose, to the contrary, that $G$ is a minimal counterexample (in terms of $|V(G)|+|E(G)|$) to this theorem. Clearly, $G$ is connected.

If $\Delta(G)\leq9$, then choose an edge $uv$ of $G$ such that $d_G(u)=\delta(G)$. Since $G$ is a 1-planar graph, $\delta(G)\leq 7$ (see \cite{FM07}). This implies that $d_G(u)+d_G(v)\leq \delta(G)+\Delta(G)\leq 16<19$. Hence configuration (a) or (b) occurs in $G$, a contradiction.

Hence, $\Delta(G)\geq 10$.
By \eqref{1} and the absence of the configuration (c), each  $d$-vertex with $2\leq d\leq 5$ (if it exists) of $G$ has a $k$-master for each $d\leq k\leq 5$.

We apply the discharging method to the associated plane graph $\gx$ of $G$. Formally, for each vertex $v\in V(\gx)$, let $c(v):=d_{\gx}(v)-6$ be its initial charge, and for each face $f\in F(\gx)$, let $c(f):=2d_{\gx}(f)-6$ be its initial charge. Clearly, $\sum_{x\in V(\gx)\cup F(\gx)}c(x)=-12<0$ by the well-known Euler's formula.

In what follows, we call a true vertex of $\gx$ \emph{big} if $d_{\gx}(G)\geq 9$, and \emph{small} if $d_{\gx}(G)\leq 8$. Since $(a)$ and $(b)$ are forbidden in $G$, any two small vertices are not adjacent in $G$.
We use $F,B$ and $S$ to represent false vertex, big vertex and small vertex, respectively, and then use these notations to represent the structure of a face of $\gx$. For example, we say that a face is an $(F,S,B,S)$-face if it is a 4-face with vertices $u_1,u_2,u_3$ and $u_4$ lying cyclically on the boundary of $f$ such that $u_1$ is false, $u_2$ is small, $u_3$ is big, and $u_4$ is small.

If a face $f\in F(\gx)$ is incident with a false vertex $u$ so that the two neighbors of $u$ in the subgraph induced by the edges of $f$ are big vertices, then $u$ is a \emph{hungry false vertex} incident with $f$.
A face in $\gx$ is \emph{burdened} if it is incident with at least one small vertex.

We define discharging rules as follows.

\begin{description}\vspace{-.6em}
  \item[R1] every big vertex of $\gx$ sends $\frac{1}{3}$ to each of its incident faces. \vspace{-.8em}
  \item[R2] every $4^+$-face of $\gx$ sends $\frac{4}{3}$ to each of its incident hungry false vertices, and $\frac{2}{3}$ to each of its incident false vertices that are not hungry. \vspace{-.8em}
  \item[R3] every false 3-face of $\gx$ sends all of its received charge after applying R1 to its incident false vertex. \vspace{-.8em}
  \item[R4] every true 3-face of $\gx$ sends all of its received charge after applying R1 to its incident small vertex (if it exists). \vspace{-.8em}
  \item[R5] every $4^+$-face of $\gx$ redistributes it remaining charge after applying R1 and R2 equitably to each of its incident small vertices (if it exists). \vspace{-.8em}
  \item[R6] every 2-vertex of $G$ receives $\frac{2}{3},\frac{1}{2},\frac{1}{2}$ and $\frac{2}{3}$ from its 2-master, 3-master, 4-master and 5-master, respectively. \vspace{-.8em}
  \item[R7] every 3-vertex of $G$ receives $\frac{1}{2},\frac{1}{2}$ and $\frac{2}{3}$ from its 3-master, 4-master and 5-master, respectively. \vspace{-2.2em}
  \item[R8] every 4-vertex of $G$ receives $\frac{1}{2}$ and $\frac{2}{3}$ from its 4-master and 5-master, respectively. \vspace{-.8em}
  \item[R9] every 5-vertex of $G$ receives $\frac{2}{3}$ from its 5-master. \vspace{-.8em}
\end{description}

Here one shall note that if $uv\in E(G)$ and $2\leq d_G(v)\leq 5$, then $u$ may simultaneously be a $k$-master of $v$ for several values $k$ with $d_G(v)\leq k\leq 5$.

Let $c'(x)$ be the charge of $x\in V(\gx)\cup F(\gx)$ after applying the above rules. Since our rules only move charge around, and do not affect
the sum, we have $$\sum_{x\in V(\gx)\cup F(\gx)}c'(x)=\sum_{x\in V(\gx)\cup F(\gx)}c(x)<0.$$
Next, we prove
that $c'(x)\geq 0$ for each $x\in V(\gx)\cup F(\gx)$. This leads to $\sum_{x\in V(\gx)\cup F(\gx)}c'(x)\geq 0$, a contradiction.

Since every $4^+$-face $f$ of $F(\gx)$ is incident with at most $\lfloor\frac{d_{\gx}(f)}{2}\rfloor$ false vertices by Lemma \ref{n.adj}(a), the charge of $f$ after applying R2 is at least $2d_{\gx}(f)-6-\frac{4}{3}\lfloor\frac{d_{\gx}(f)}{2}\rfloor>0$ for $d_{\gx}(f)\geq 5$. On the other hand, if $f$ is a 4-face incident with at least one hungry false vertex, then it is incident with at least two big vertices and thus $c'(f)\geq 2\times 4-6+2\times \frac{1}{3}-2\times \frac{4}{3}=0$ by R1 and R2, and if $f$ is a 4-face incident with none hungry false vertex, then $c'(f)\geq 2\times 4-6-2\times \frac{2}{3}>0$ by R2.
Hence, R1--R5 guarantee that $c'(f)\geq 0$ for each $f\in F(\gx)$.

By R1, R3 and R4, it is easy to conclude the following three claims.

\begin{clm}\label{clm1}
Every $(F,B,B)$-face sends $\frac{2}{3}$ to its incident false vertex. \hfill$\square$
\end{clm}

\begin{clm}\label{clm2}
Every $(F,B,S)$-face sends $\frac{1}{3}$ to its incident false vertex. \hfill$\square$
\end{clm}

\begin{clm}\label{clm3}
Every burdened true 3-face  sends $\frac{2}{3}$ to its incident small vertex. \hfill$\square$
\end{clm}

Now we consider burdened $4^+$-faces.

\begin{clm}\label{clm4}
Every burdened 4-face sends to each of its incident small vertices $\frac{1}{3}$ if $f$ is an $(F,S,F,S)$-face, $\frac{5}{6}$ if $f$ is an $(F,S,B,S)$-face, $1$ if $f$ is an $(F,S,F,B)$-face,
and at least $\frac{4}{3}$ otherwise.
\end{clm}

\begin{proof}
If $f$ is an $(F,S,F,S)$-face, then the false vertices incident with $f$ are not hungry, and thus by R2 and R5, $f$ sends $\frac{1}{2}\times(2\times 4-6-2\times\frac{2}{3})=\frac{1}{3}$ to each of its incident small vertices.

If $f$ is an $(F,S,B,S)$-face, then the false vertex incident with $f$ is not hungry, and thus by R1, R2 and R5, $f$ sends $\frac{1}{2}\times(2\times 4-6+\frac{1}{3}-\frac{2}{3})=\frac{5}{6}$ to each of its incident small vertices.

If $f$ is an $(F,S,F,B)$-face, then the false vertices incident with $f$ are not hungry, and thus by R1, R2 and R5, $f$ sends $2\times 4-6+\frac{1}{3}-2\times\frac{2}{3}=1$ to its incident small vertex.

By symmetry, $f$ can be of another types among $(S,B,B,B),(S,B,S,B),(F,S,B,B)$ and $(F,B,S,B)$. In each case we can similarly calculate that $f$ sends at least $\frac{4}{3}$ to each of its incident small vertices.
\end{proof}

\begin{clm}\label{clm5}
Every burdened $5^+$-face sends at least $\frac{4}{3}$ to each of its incident small vertices.
\end{clm}

\begin{proof}

If $f$ is not incident with hungry false vertex, then $f$ is incident with at most $\lfloor\frac{d_{\gx}(f)}{2}\rfloor$ false vertices and at most $\lfloor\frac{d_{\gx}(f)}{2}\rfloor$ small vertices. Hence $f$ sends at least $(2d_{\gx}(f)-6-\frac{2}{3}\lfloor\frac{d_{\gx}(f)}{2}\rfloor)/\lfloor\frac{d_{\gx}(f)}{2}\rfloor\geq \frac{4}{3}$ to each of its incident small vertices by R2 and R5.

If $f$ is incident with a hungry false vertex, then $f$ is incident with at most $\lfloor\frac{d_{\gx}(f)}{2}\rfloor-1$ hungry false vertices (otherwise $f$ is not burdened) and at most $\lceil\frac{d_{\gx}(f)-3}{2}\rceil$ small vertices. By R1, R2 and R5, $f$ sends at least
$(2d_{\gx}(f)-6-\frac{4}{3}(\lfloor\frac{d_{\gx}(f)}{2}\rfloor-1)-\frac{2}{3})/\lceil\frac{d_{\gx}(f)-3}{2}\rceil\geq \frac{4}{3}$ to each of its incident small vertices.
\end{proof}

Now we calculate the final charge of each vertex $v\in V(\gx)$.

\textbf{Case 1.} $v$ is a false vertex.

If $v$ is incident with at least three $(F,B,B)$-faces, then by Claim \ref{clm1}, $c'(v)\geq 4-6+3\times\frac{2}{3}=0$.

If $v$ is incident with exactly two $(F,B,B)$-faces, then each of another two faces that are incident with $v$ is an $(F,B,S)$-face or a $4^+$-face. Hence by Claims \ref{clm1}, \ref{clm2} and R2, we have $c'(v)\geq 4-6+2\times \frac{2}{3}+2\times\min\{\frac{1}{3},\frac{2}{3}\}=0$.

If $v$ is incident with exactly one $(F,B,B)$-face, then $v$ is incident with at least one $4^+$-face, because otherwise $v$ is incident with an $(F,S,S)$-face, which is impossible since small vertices are not adjacent in $G$. Under this condition, by Claims \ref{clm1}, \ref{clm2} and R2, we have $c'(v)\geq 4-6+\frac{2}{3}+2\times\min\{\frac{1}{3},\frac{2}{3}\}+\frac{2}{3}=0$.

If $v$ is incident with none $(F,B,B)$-face, then $v$ is incident with at least two $4^+$-faces, because otherwise $v$ is incident with an $(F,S,S)$-face, which is impossible since small vertices are not adjacent in $G$. Under this condition, by Claims \ref{clm1}, \ref{clm2} and R2, we have $c'(v)\geq 4-6+2\times\min\{\frac{1}{3},\frac{2}{3}\}+2\times \frac{2}{3}=0$.

\textbf{Case 2.} $v$ is a 2-vertex.

By Lemma \ref{n.adj}(b), $v$ is not incident with a false 3-face, and by R6, $v$ receives $\frac{2}{3},\frac{1}{2},\frac{1}{2}$ and $\frac{2}{3}$ from its 2-master, 3-master, 4-master and 5-master, respectively.

If $v$ is incident with a true 3-face, then $v$ is adjacent to two big vertices in $\gx$, and the other face $f$ incident with $v$ is either a $5^+$-face, or an $(F,B,S,B)$-face, or a $(B,B,S,B)$-face, or a $(S,B,S,B)$-face. In either case, $f$ sends at least $\frac{4}{3}$ to $v$ by Claims \ref{clm4} and \ref{clm5}. Hence $c'(v)\geq 2-6+\frac{2}{3}+\frac{4}{3}+\frac{2}{3}+\frac{1}{2}+\frac{1}{2}+\frac{2}{3}>0$ by
Claim \ref{clm3}.

If $v$ is incident with two $4^+$-faces, one of which is a $5^+$-face, then $c'(v)\geq 2-6+\frac{1}{3}+\frac{4}{3}+\frac{2}{3}+\frac{1}{2}+\frac{1}{2}+\frac{2}{3}=0$ by Claims \ref{clm4} and \ref{clm5}.

If $v$ is incident with two $4$-faces, then none of the two 4-faces incident with $v$ is an $(F,S,F,S)$-face (otherwise a multi-edge appears in $G$). This implies
$c'(v)\geq 2-6+2\times\frac{5}{6}+\frac{2}{3}+\frac{1}{2}+\frac{1}{2}+\frac{2}{3}=0$ by Claim \ref{clm4}.

\textbf{Case 3.} $v$ is a 3-vertex.

By R7, $v$ receives $\frac{1}{2},\frac{1}{2}$ and $\frac{2}{3}$ from its 3-master, 4-master and 5-master, respectively.

If $v$ is incident with a $5^+$-face, then $c'(v)\geq 3-6+\frac{4}{3}+\frac{1}{2}+\frac{1}{2}+\frac{2}{3}=0$ by Claim \ref{clm5}.

If $v$ is incident with three $4$-faces, then at most one of them is an $(F,S,F,S)$-face (otherwise two small vertices are adjacent in $G$). Therefore,
$c'(v)\geq 3-6+\frac{1}{3}+2\times\frac{5}{6}+\frac{1}{2}+\frac{1}{2}+\frac{2}{3}>0$
by Claim \ref{clm4}.

If $v$ is incident with two 4-faces and one 3-face, then the two 4-faces incident with $v$ cannot be both of $(F,S,F,S)$-type. If none of them is of $(F,S,F,S)$-type, then
$c'(v)\geq 3-6+2\times\frac{5}{6}+\frac{1}{2}+\frac{1}{2}+\frac{2}{3}>0$ by Claim \ref{clm4}. If one of them is of type $(F,S,F,S)$, then the other one is of type $(F,B,B,S)$. This implies $c'(v)\geq 3-6+\frac{1}{3}+\frac{4}{3}+\frac{1}{2}+\frac{1}{2}+\frac{2}{3}>0$
by Claim \ref{clm4}.

If $v$ is incident with one $4$-face and two 3-faces, then the $4$-face incident with $v$ is not of $(F,S,F,S)$-type (otherwise a multi-edge occurs in $G$). If $v$ is incident with a true 3-face, then $c'(v)\geq 3-6+\frac{2}{3}+\frac{5}{6}+\frac{1}{2}+\frac{1}{2}+\frac{2}{3}>0$ by Claims \ref{clm3} and \ref{clm4}. If $v$ is incident with two false 3-faces, then by Lemmas \ref{n.adj}(c) and \ref{n.adj}(d), $v$ is adjacent to two false vertices and incident with a $5^+$-face, which is impossible in this case.

If $v$ is incident with three 3-faces, then by Lemma \ref{n.adj}(d), all of those 3-faces are true. This implies $c'(v)\geq 3-6+3\times \frac{2}{3}+\frac{1}{2}+\frac{1}{2}+\frac{2}{3}>0$ by Claim \ref{clm3}.

\textbf{Case 4.} $v$ is a true 4-vertex.

By R8, $v$ receives $\frac{1}{2}$ and $\frac{2}{3}$ from its 4-master and 5-master, respectively.

If $v$ is incident with at least one $5^+$-face, then $c'(v)\geq 4-6+\frac{4}{3}+\frac{1}{2}+\frac{2}{3}>0$ by Claim \ref{clm5}. Therefore we assume that $v$ is incident only with $4^-$-faces.

If $v$ is incident with four 3-faces, then at least two of them are true ones (otherwise two false vertices are adjacent in $\gx$ or there exists a multi-edge in $G$). Hence $c'(v)\geq 4-6+2\times\frac{2}{3}+\frac{1}{2}+\frac{2}{3}>0$ by Claim \ref{clm3}.

If $v$ is incident with at least three $4$-faces, then $c'(v)\geq 4-6+3\times\frac{1}{3}+\frac{1}{2}+\frac{2}{3}>0$ by Claim \ref{clm4}.

If $v$ is incident with at exactly two $4$-faces, then at least one of them is not of $(F,S,F,S)$-type, which implies $c'(v)\geq 4-6+\frac{1}{3}+\frac{5}{6}+\frac{1}{2}+\frac{2}{3}>0$ by Claim \ref{clm4}.

If $v$ is incident with exactly one $4$-face and this 4-face is not of $(F,S,F,S)$-type, then $c'(v)\geq 4-6+\frac{5}{6}+\frac{1}{2}+\frac{2}{3}=0$ by Claim \ref{clm4}.

If $v$ is incident with one $(F,S,F,S)$-face and three 3-faces, then $v$ is incident with a true 3-face. This implies that $c'(v)\geq 4-6+\frac{1}{3}+\frac{2}{3}+\frac{1}{2}+\frac{2}{3}>0$ by Claims \ref{clm3} and \ref{clm4}.

\textbf{Case 5.} $v$ is a 5-vertex.

By R9, $v$ receives $\frac{2}{3}$ from its 5-master.

If $v$ is incident with at least one $4^+$-face, then $c'(v)\geq 5-6+\frac{1}{3}+\frac{2}{3}=0$ by Claim \ref{clm4}.

If $v$ is incident with five 3-faces, then at least one of them is true, which implies $c'(v)\geq 5-6+\frac{2}{3}+\frac{2}{3}>0$ by Claim \ref{clm3}.

\textbf{Case 6.} $v$ is a vertex of degree between 6 and 14.

By the absence of the configuration (a), every $5^-$-vertex is adjacent only to $15^+$-vertex in $G$. Therefore, $v$ cannot be a master of any vertex. If $v$ is a small vertex, then $v$ does not give out any charge by R1--R9, and thus $c'(v)=c(v)=d_{\gx}(v)-6\geq 0$. If $v$ is a big vertex, that is, $d_{\gx}(v)\geq 9$, then by R1, $c'(v)\geq d_{\gx}(v)-6-\frac{1}{3}d_{\gx}(v)=\frac{1}{3}(2d_{\gx}(v)-18)\geq 0$.

\textbf{Case 7.} $v$ is a 15-vertex.

By the absence of the configuration (a), $v$ is adjacent only to $5^+$-vertex in $G$. Therefore, by \eqref{2}, $v$ can be a 5-master of at most four vertices , and cannot be a 4-master, or a 3-master, or a 2-master of any vertex. By R1 and R9,
$c'(v)\geq 15-6-\frac{1}{3}\times 15-4\times \frac{2}{3}>0$.

\textbf{Case 8.} $v$ is a 16-vertex.

By the absence of the configuration (a), $v$ is adjacent only to $4^+$-vertex in $G$. Therefore, by \eqref{2}, $v$ can be a 5-master of at most four vertices, a 4-master of at most three vertices, and cannot be a 3-master or a 2-master of any vertex. By R1, R8 and R9,
$c'(v)\geq 16-6-\frac{1}{3}\times 16-4\times \frac{2}{3}-3\times \frac{1}{2}>0$.

\textbf{Case 9.} $v$ is a 17-vertex.

By the absence of the configuration (a), $v$ is adjacent only to $3^+$-vertex in $G$. Therefore, by \eqref{2}, $v$ can be a 5-master of at most four vertices, a 4-master of at most three vertices, a 3-master of at most two vertices, and cannot be a 2-master of any vertex. By R1, R7, R8 and R9,
$c'(v)\geq 17-6-\frac{1}{3}\times 17-4\times \frac{2}{3}-3\times \frac{1}{2}-2\times\frac{1}{2}>0$.

\textbf{Case 10.} $v$ is a $18^+$-vertex.

By \eqref{2}, $v$ can be a 5-master of at most four vertices, a 4-master of at most three vertices, a 3-master of at most two vertices, and a 2-master of at most one vertex. By R1, R6, R7, R8 and R9,
$c'(v)\geq d_{\gx}(v)-6-\frac{1}{3} d_{\gx}(v)-4\times \frac{2}{3}-3\times \frac{1}{2}-2\times\frac{1}{2}-\frac{2}{3}=\frac{1}{6}(4d_{\gx}(v)-71)>0$.
\end{proof}

%
%

\section{List edge and list total coloring}

A \emph{critical edge $M$-choosable graph} (resp.\,\emph{critical }\emph{total $(M+1)$-choosable graph}) is a graph with maximum degree at most $M$ such that $G$ is not edge $M$-choosable (resp.\,total $(M+1)$-choosable), and any proper subgraph of $G$ is edge $M$-choosable (resp.\,total $(M+1)$-choosable). The structures of such critical graphs were investigated by Wu and Wang \cite{WW08}, who proved the following two useful results.

\begin{lem}\cite[Lemma 2.2]{WW08}\label{ww-lem-1}
If $G$ is a critical edge $M$-choosable graph (resp.\,critical total $(M+1)$-choosable graph), then for every edge $xy\in E(G)$ with $d_G(x)\leq \lfloor\frac{M}{2}\rfloor$, we have $d_G(x)+d_G(y)\geq M+2$.
\end{lem}

\begin{lem}\cite[Lemma 2.3]{WW08}\label{ww-lem-2}
If $G$ is a critical edge $M$-choosable graph (resp.\,critical total $(M+1)$-choosable graph), then there is no $k$-alternator $F$ in $G$ for any integer $2\leq k\leq \lfloor\frac{M}{2}\rfloor$.
\end{lem}

Now we apply the above two lemmas along with Theorem \ref{str-1-g3} to proving the following theorem.

\begin{thm}\label{main-thm-1}
  If $G$ is a 1-planar graph with maximum degree $\Delta\geq 18$, then $\chi'(G)=\chi'_l(G)=\Delta$ and $\chi''(G)=\chi''_l(G)=\Delta+1$.
\end{thm}

\begin{proof}
Let $M$ be an integer such that $\Delta\leq M$ and $M\geq 18$.
It is sufficient to prove that $\chi'_l(G)\leq M$ and $\chi''_l(G)\leq M+1$.

Suppose, to the contrary, that there is a critical edge $M$-choosable graph (resp.\,critical total $(M+1)$-choosable graph) $G$. By Lemma \ref{ww-lem-1}, $\delta(G)\geq 2$.
Since $G$ is a 1-planar graph, by Theorem \ref{str-1-g3}, $G$ contains either (i) an edge $xy$ with $d_G(x)\leq 8<\lfloor\frac{M}{2}\rfloor$ and $d_G(x)+d_G(y)\leq 19\leq M+1$, or (ii) a $k$-alternator for some $k\in \{2,3,4,5\}$. However, Lemma \ref{ww-lem-1} implies that the local configuration (i) is forbidden, and  Lemma \ref{ww-lem-2} implies that the local configuration (ii) is absent. This contradiction completes the proof.
\end{proof}

%

\section{$(p,1)$-total labelling}

A \emph{critical $(p,1)$-total $k$-labelled graph} is a graph $G$ such that it admits no $(p,1)$-total $k$-labelling, and any proper subgraph of $G$ has a $(p,1)$-total $k$-labelling. Zhang, Yu and Liu \cite{ZYL11} proved the following two structural theorems for the critical $(p,1)$-total labelled graph.

\begin{lem}\cite[Lemmas 2.1 and 2.2]{ZYL11}\label{zyl-lem-1}
Let $G$ be a critical $(p,1)$-total $(M+2p-2)$-labelled graph with maximum degree at most $M$. For any edge $uv\in E(G)$, if $\min\{d_G(u),d_G(v)\}\leq \lfloor\frac{M+2p-2}{2p}\rfloor$, then $d_G(u)+d_G(v)\geq M+2$, and otherwise, $d_G(u)+d_G(v)\geq M-2p+3$.
\end{lem}

\begin{lem}\cite[Lemma 2.4]{ZYL11}\label{zyl-lem-2}
If $G$ is a critical $(p,1)$-total $(M+2p-2)$-labelled graph with maximum degree at most $M$, then there is no $k$-alternator $F$ in $G$ for any integer $2\leq k\leq \lfloor\frac{M+2p-2}{2p}\rfloor$.
\end{lem}

\begin{thm}\label{main-thm-2}
  If $G$ is a 1-planar graph with $\Delta(G)\geq 8p+2$ and $p\geq 2$, then $\lambda^T_p(G)\leq\Delta(G)+2p-2$.
\end{thm}

\begin{proof}
Let $M$ be an integer such that $\Delta(G)\leq M$ and $M\geq 8p+2\geq 18$. Now, proving $\lambda^T_p(G)\leq M+2p-2$ is sufficient.
Suppose, to the contrary, that $G$ is a critical $(p,1)$-total $(M+2p-2)$-labelled graph. By Lemma \ref{zyl-lem-1}, $\delta(G)\geq 2$.
Since $G$ is a 1-planar graph, by Theorem \ref{str-1-g3}, $G$ contains either (i) an edge $xy$ with $d_G(x)\leq 5\leq \lfloor\frac{M+2p-2}{2p}\rfloor$ and $d_G(x)+d_G(y)\leq 19\leq M+1$, or (b) an edge with $d_G(x)+d_G(y)\leq 16\leq M-2p+2$, or (c) a $k$-alternator for some $2\leq k\leq 5\leq \lfloor\frac{M+2p-2}{2p}\rfloor$. However, the configuration (i) or (ii) cannot appear in $G$ by Lemma \ref{zyl-lem-1}, and  the configuration (iii) is absent from $G$ by Lemma \ref{zyl-lem-2}.
\end{proof}

\section{Equitable edge coloring}

A \emph{critical equitable edge $M$-colorable graph} is a graph $G$ such that $G$ admits no equitable edge $M$-colorings, and
any proper subgraph $H$ of $G$ is equitable edge $M$-colorable. The following are two useful structural results for the critical equitable edge $k$-colorable graph.

\begin{lem}\cite[Lemma 6]{HWYZ17}\label{HWYZ-lem-1}
If $G$ is a critical equitable edge $M$-colorable graph, then $d_G(x)+d_G(y)\geq M+2$ for any $xy\in E(G)$.
\end{lem}


\begin{lem}\cite[Lemma 7]{HWYZ17}\label{HWYZ-lem-2}
If $G$ is a critical equitable edge $M$-colorable graph, then there is no $k$-alternating subgraph $F$ in $G$ for any integer $2\leq k\leq \lfloor\frac{M}{2}\rfloor$.
\end{lem}

\noindent \textbf{Remark:}  the original statements of Lemmas 6 and 7 in \cite{HWYZ17} are not as the same as the above two ones. Actually, Lemma 6 of the paper \cite{HWYZ17} states that if $G$ is a critical equitable edge $M$-colorable graph with $M\geq 21$, then $d_G(x)+d_G(y)\geq 23$ for any $xy\in E(G)$. Indeed, the proof there is still applicable for proving Lemma \ref{HWYZ-lem-1} here, only with few changes. On the other hand, from the fourth paragraph to the end of the proof of Lemma 7 in \cite{HWYZ17}, the authors claim that any critical equitable edge $M$-colorable graph does not contains a
bipartite subgraph $H'$ with partite sets $X'', Y'$ such that $d_{H'}(x)=d_G(x)\leq k$ for each $x\in X''$, and $d_{H'}(y)\geq k$ for each $y\in Y'$, where $2\leq k\leq 5$. One can easily check that their proof can be directly extended to the case when $2\leq k\leq \lfloor\frac{M}{2}\rfloor$, without changing any word. Therefore, there is no $k$-alternating subgraph in a critical equitable edge $M$-colorable graph $G$ for any integer $2\leq k\leq \lfloor\frac{M}{2}\rfloor$.

\begin{thm}\label{main-thm-2}
  If $G$ is a 1-planar graph, then $\chi'_{\equiv}(G)\leq 18$.
\end{thm}

\begin{proof}
Let $M$ be an integer such that $M\geq 18$. We just need to prove that $G$ has an equitable edge $M$-coloring. Suppose, to the contrary, that $G$ is a critical equitable edge $M$-colorable graph. By Lemma \ref{HWYZ-lem-1}, $\delta(G)\geq 2$. Since $G$ is a 1-planar graph, by Theorem \ref{str-1-g3}, $G$ contains either (i) an edge $xy$ with $d_G(x)+d_G(y)\leq 19$, or (ii) a $k$-alternating subgraph for some $k\in \{2,3,4,5\}$. However, $d_G(x)+d_G(y)\geq M+2\geq 20$ for any $xy\in E(G)$ by Lemma \ref{HWYZ-lem-1}, which makes the configuration (i) absent, and Lemma \ref{HWYZ-lem-2} do not support the appearance of the configuration (ii).
\end{proof}

\bibliographystyle{srtnumbered}
\bibliography{mybib}

\end{document}